\UseRawInputEncoding% AMS-LaTeX 1.2
\documentclass[11pt]{amsart}
\usepackage{nccmath}

\usepackage{amsmath,amsthm,amssymb}
\usepackage{times}
\usepackage{enumerate}
\usepackage{color}
\newcommand{\ubar}{\underline}

\newtheorem{prop}{Proposition}
\newtheorem{theo}[prop]{Theorem}
\newtheorem{lemm}[prop]{Lemma}

\newtheorem{rmk}[prop]{Remark}

\newcommand{\be}{\begin{equation}}
\newcommand{\ee}{\end{equation}}
\newcommand{\lt}{\left}
\newcommand{\rt}{\right}
\newcommand{\goto}{\rightarrow}
\newcommand{\al}{\alpha}
\newcommand{\e}{\epsilon}
\renewcommand{\leq}{\leqslant}
\renewcommand{\geq}{\geqslant}

\newcommand{\R}{\mathbb{R}}
\newcommand{\M}{\mathcal{M}}
\newcommand{\ka}{\kappa}
\newcommand{\la}{\lambda}

\newcommand{\s}{\sigma}
\newcommand{\ga}{\gamma}
\newcommand{\p}{\partial}
\newcommand{\uus}{\bar{u}^*}
\newcommand{\lus}{\ubar{u}^*}

\newcommand{\gas}{\gamma^*}

\newcommand{\dS}{\mathbb{S}}
\newcommand{\m}{\mathfrak{m}}

\numberwithin{equation}{section}

\begin{document}
\setlength{\baselineskip}{1.2\baselineskip}

\title[Entire self-expanders in Minkowski space]
{Entire self-expanders for power of $\sigma_k$ curvature flow in Minkowski space}

\author{Zhizhang Wang}
\address{School of Mathematical Science, Fudan University, Shanghai, China}
\email{zzwang@fudan.edu.cn}
\author{Ling Xiao}
\address{Department of Mathematics, University of Connecticut,
Storrs, Connecticut 06269}
\email{ling.2.xiao@uconn.edu}
\thanks{Research of the first author is  sponsored by Natural Science  Foundation of Shanghai, No.20JC1412400,  20ZR1406600 and supported by NSFC Grants No.11871161,12141105.}

\begin{abstract}
In \cite{WX22}, we prove that if an entire, spacelike, convex hypersurface $\M_{u_0}$ has bounded principal curvatures, then the $\sigma_k^{1/\al}$ (power of $\s_k$) curvature flow
starting from $\M_{u_0}$ admits a smooth convex solution $u$ for $t>0.$ Moreover, after rescaling, the flow converges to a convex self-expander $\tilde{\M}=\{(x, \tilde{u}(x))\mid x\in\R^n\}$ that satisfies $\s_k(\ka[\tilde{\M}])=(-\lt<X_0, \nu_0\rt>)^\al.$ Unfortunately, the existence of self-expander for power of $\sigma_k$ curvature flow in Minkowski space has not been studied before. In this paper, we fill the gap.

\end{abstract}

\maketitle
\section{Introduction}
Let $\R^{n, 1}$ be the Minkowski space with the Lorentzian metric
\[ds^2=\sum_{i=1}^{n}dx_{i}^2-dx_{n+1}^2.\]
In this paper, we will devote ourselves to the study of spacelike hypersurfaces $\M$ with prescribed
$\sigma_k$ curvature in Minkowski space $\R^{n, 1}$. Here, $\s_k$ is the $k$-th elementary symmetric polynomial, i.e.,
 \[\s_k(\ka)=\sum\limits_{1\leq i_1<\cdots<i_k\leq n}\ka_{i_1}\cdots\ka_{i_k}.\]
Any such hypersurface $\M$ can be written locally as a graph of a function
$x_{n+1}=u(x), x\in\R^n,$ satisfying the spacelike condition
\be\label{int1.1}
|Du|<1.
\ee
More specifically, we will study self-similar solutions of
flow by powers of the $\s_k$ curvature. Namely, we are interested in entire, spacelike, convex hypersurfaces which move under $\s_k$ curvature flows by homothety.

Let $X(\cdot, t)$ be a spacelike, strictly convex solution of
\be\label{hs1}
\frac{\p X}{\p t}(p, t)=\s_k^\beta(p, t)\nu(p, t)
\ee
for some $\beta\in(0, \infty).$ If the hypersurfaces $\M(t)$ given by $X(\cdot, t)$ move homothetically, then
$X(\cdot, t)=\phi(t)X_0$ for some positive function $\phi.$ Since the normal vector field is unchanged by homotheties,
by taking the inner product of \eqref{hs1} with $\nu_0=\nu(\cdot, t)$ we obtain
\[\phi'\lt<X_0, \nu_0\rt>=-\s_k^\beta(\ka[X_0])\phi^{-k\beta},\]
where $\ka[X_0]=(\ka_1, \cdots, \ka_n)$ is the principal curvatures of $\M_0$ at $X_0.$
Therefore, we must have
\[\phi'(t)\phi^{k\beta}(t)=\lambda\]
and
\[\s_k^\beta=-\lambda\lt<X_0, \nu_0\rt>.\]
In this paper, we will consider the case when $\lambda>0,$ which we call {\it expanding solutions}. Through rescaling, we may also assume $\lambda=1.$

Complete noncompact self similar solutions of curvature flows in Euclidean space have been studied intensively (for example \cite{AW94, CSS07, CDL21, HIMW19, SW10, SX20, Urb98}).
However, in Minkowski space, there is no corresponding known result yet.

It is well-known that the hyperboloid is a self-expander. In \cite{WXflow}, we have proved the rescaled convex curvature flows, including Gauss curvature flow, converge to the hyperboloid. Therefore, a natural question to ask is whether there exist self-expanders other than the hyperboloid? Moreover, if such self-expanders exist, can we construct some curvature flows such that their rescaled flows converge to these new self-expanders? In this paper and an upcoming paper \cite{WX22}, we give affirmative answers to both questions.

Now consider $\M_{u_0}=\{(x, u_0(x))\mid x\in\R^n\},$ an entire, spacelike, convex hypersurface satisfying
$u_0(x)-|x|\goto\varphi\lt(\frac{x}{|x|}\rt)$ as $|x|\goto\infty.$ By translating $\M_{u_0}$ vertically we may also assume $\varphi\lt(\frac{x}{|x|}\rt)>0.$
In an upcoming paper \cite{WX22}, we prove that, if in addition $\M_{u_0}$ also has bounded principal curvatures, then the equation
\[
\lt\{
\begin{aligned}
\frac{\p X}{\p t}&=\s_k^{1/\al}\nu\\
X(x, 0)&=\M_{u_0},
\end{aligned}
\rt.
\]
where $\al\in(0, k],$ admits a smooth convex solution $u$ for $t>0.$ Moreover, after rescaling the flow converges to a convex self-expander $\tilde{\M}=\{(x, \tilde{u}(x))\mid x\in\R^n\}$ that satisfies
\be\label{int0.1}
\s_k(\ka[\tilde{\M}])=(-\lt<X, \nu\rt>)^\al
\ee
and
\be\label{int0.2}
\tilde{u}-|x|\goto\varphi\lt(\frac{x}{|x|}\rt)\,\,\mbox{as $|x|\goto\infty.$}
\ee
Unfortunately, the existence of solutions of equations \eqref{int0.1} and \eqref{int0.2} have not been studied before. In this paper, we fill the gap and prove the following theorems.
\begin{theo}
\label{int-thm1}
Suppose $\varphi$ is a positive $C^2$ function defined on $\mathbb{S}^{n-1},$ i.e., $\varphi\in C^2(\mathbb{S}^{n-1}).$ Then there exists a unique, entire, strictly convex, spacelike hypersurface
$\M_u=\{(x, u(x))\mid x\in\R^n\}$ satisfying
\be\label{int0.1'}
\s_n(\ka[\M_u])=(-\lt<X, \nu\rt>)^\al\,\,\mbox{for any $\al\in(0, n],$}
\ee
and
\be\label{int0.2'}
u(x)-|x|\goto\varphi\lt(\frac{x}{|x|}\rt)\,\,\mbox{as $|x|\goto\infty.$}
\ee
\end{theo}
\begin{rmk}\label{rmk1}
Note that, unlike previously known results on spacelike hypersurfaces with prescribed Gauss curvature (see \cite{BP, GJS, AML, RWX20}), the right hand side of
\eqref{int0.1} is unbounded. Therefore, the proof of Theorem \ref{int-thm1} is different from earlier works, and we need to develop new techniques to prove it.
\end{rmk}
Using the solution we obtained in Theorem \ref{int-thm1} as a subsolution we can also prove
\begin{theo}
\label{int-thm2}
Suppose $\varphi$ is a positive $C^2$
function defined on $\mathbb{S}^{n-1},$ i.e., $\varphi\in C^2(\mathbb{S}^{n-1}).$ Then there exists a unique, entire, strictly convex, spacelike hypersurface
$\M_u=\{(x, u(x))\mid x\in\R^n\}$ satisfying
\be\label{int0.1'}
\s_k(\ka[\M_u])=(-\lt<X, \nu\rt>)^\al\,\,\mbox{for any $\al\in(0, k],$}
\ee
and
\be\label{int0.2'}
u(x)-|x|\goto\varphi\lt(\frac{x}{|x|}\rt)\,\,\mbox{as $|x|\goto\infty,$}
\ee
where $k\leq n-1.$
\end{theo}

The paper is organized as follows. In Section \ref{gcf}, we prove Theorem \ref{int-thm1}.
In particular, we develop new techniques to prove the local estimates.
In Section \ref{skc}, combining the result we obtained in Section \ref{gcf} with our ideas developed in \cite{RWX20} and \cite{WX20},
we prove Theorem \ref{int-thm2}. The arguments in this section are modifications of our arguments in \cite{RWX20} and \cite{WX20}.

\section{Gauss curvature self-expander}
\label{gcf}
In this section, we want to show there exists an entire, strictly spacelike, convex solution to the following equation
\be\label{gcf1}
\s_n(\ka[\M_u])=\lt(-\lt<X, \nu\rt>\rt)^\al,\,\,0<\al\leq n,
\ee
and
\be\label{gcf1*}
u(x)-|x|\goto\varphi\lt(\frac{x}{|x|}\rt)\,\,\mbox{as $|x|\goto\infty,$}
\ee
where $\varphi$ is a positive function defined on $\dS^{n-1}.$
Let $u^*$ be the Legendre transform of $u.$ By Section 3 of \cite{AML} and Lemma 14 in \cite{WX20}, we know when $u$ is a solution of \eqref{gcf1} and \eqref{gcf1*}, then
$u^*$ satisfies the following PDE,
\be\label{gcf2}
\left\{
\begin{aligned}
\s_n(D^2u^*)&=\frac{w^{*\al-n-2}}{(-u^*)^\al}\,\,&\mbox{in $B_1$}\\
u^*&=\varphi^*(\xi)\,\,&\mbox{on $\p B_1,$}
\end{aligned}
\right.
\ee
where $\varphi^*=-\varphi<0$ on $\p B_1$ and $w^*=\sqrt{1-|\xi|^2}.$
Since \eqref{gcf2} is a degenerate equation, we will study the following approximate problem instead
\be\label{gcf-approx}
\left\{
\begin{aligned}
\s_n(D^2u^*)&=\frac{(1-s|\xi|^2)^{(\al-n-2)/2}}{(-u^*)^\al}\,\,&\mbox{in $B_1$}\\
u^*&=\varphi^*(\xi)\,\,&\mbox{on $\p B_1,$}
\end{aligned}
\right.
\ee
where $0<s<1.$
\begin{rmk}
\label{gcf-remk1}
In this remark, we want to explain why $\al$ needs to be less than or equal to $n.$ Note that here we want to construct an entire solution to equation \eqref{gcf1}. This requires
$|Du^*(\xi)|\goto\infty$ as $\xi\goto\p B_1.$ One can see that if $u^*$ is a solution to \eqref{gcf2} then we have
\[
\begin{aligned}
&\int_{B_1}\det D^2u^*=\int_{Du^*(B_1)} 1\\
&\sim\int_{B_1}(1-|\xi|^2)^{\frac{\al-n-2}{2}}\sim\int_{1/2}^1r^{n-1}(1-r^2)^{\frac{\al-n-2}{2}}dr\sim\int_{1/2}^1(1-r^2)^{\frac{\al-n-2}{2}}dr^2.\\
\end{aligned}
\]
Since $Du^*(B_1)=\R^n,$ we know $\int_{1/2}^1(1-r^2)^{\frac{\al-n-2}{2}}dr^2$ blows up, which implies $\al\leq n.$
\end{rmk}
\subsection{Solvability of equation \eqref{gcf-approx}}
We will show there exists a solution $u^{s*}$ of \eqref{gcf-approx} for $0<s<1$. For our convenience, in the following, when there is no confusion, we will drop the superscript $s$ and denote
$u^{s*}$ by $u^*.$
\begin{lemm}(\textbf{$C^0$ estimate for $u^*$})
\label{c0-gcf-approx-lem}
Let $u^*$ be the solution of \eqref{gcf-approx}, then
\be\label{gfc3}
|u^*|<C,
\ee
where $C=C(|\varphi^*|_{C^0}).$
\end{lemm}
\begin{proof}
Let $-C_0=\max\limits_{\xi\in\p B_1}\varphi^*<0,$ by the convexity of $u^*$ we know $-C_0>u^*$ in $B_1.$
On the other hand, \cite{AML} proves that there exists a solution $\lus$ satisfies
\be\label{gcf4}
\left\{
\begin{aligned}
\s_n(D^2\lus)&=\frac{1}{K}(1-|\xi|^2)^{-\frac{n+2}{2}}\,\,&\mbox{in $B_1$}\\
\lus&=\varphi^*(\xi)\,\,&\mbox{on $\p B_1,$}
\end{aligned}
\right.
\ee
for any $K\in\R_+.$

Now let $K\leq C_0^\al$ we have $\s_n(D^2\lus)>\s_n(D^2u^*)$ in $B_1,$
and $\lus=u^*$ on $\p B_1$. By the maximum principle we obtain
\be\label{gcf4*}-C_0>u^*>\lus.\ee
\end{proof}

Following \cite{CNS1}, we can obtain the $C^1$  and $C^2$ estimates for the solution of \eqref{gcf-approx}. Applying the method of continuity, we get the solvability of \eqref{gcf-approx}. Therefore, in the following, we will focus on establishing local estimates for $u^{s*}$.

\subsection{Local $C^1$ estimates}
This subsection contains two parts. In the first part, we will prove $(1-s|\xi|^2)|Du^{s*}(\xi)|<C,$ where $C$ is independent of $s$ and $\xi.$ This estimate will be useful for obtaining local $C^2$ estimates in the next subsection. In the second part, we will show $|Du^{s^*}|(\xi)\goto\infty$ as $s, |\xi|\goto 1.$ This is to illustrate that the Lendrengre transform of $u^{s*}$, denoted by $u^s,$ converges to an entire solution of \eqref{int0.1} as $s\goto 1$.

\subsubsection{Local $C^1$ upper bound} In this part we will show  $(1-s|\xi|^2)|Du^{s*}(\xi)|<C,$ where $C$ is a constant independent of $s.$

First, applying \cite{CNS1} we know that we can solve the following equation
\be\label{gcf1.1}
\left\{
\begin{aligned}
\s_n(D^2u^*)&=\frac{1}{C_1}<\frac{1}{(-\min\lus)^\al}\,\,&\mbox{in $B_1$}\\
u^*&=\varphi^*(\xi)\,\,&\mbox{on $\p B_1.$}
\end{aligned}
\right.
\ee
We will denote the solution to \eqref{gcf1.1} by $u_0^{*}.$ It's clear that $u_0^{*}>u^{s*}>\lus,$ where $\lus$ is the solution to \eqref{gcf4}.

Next, denote $h^s:=1-s|\xi|^2$ and $V^s=|Du^{s*}|,$ we prove
\begin{lemm}
\label{loc-c1-lem1}
For any $s\in [\frac{1}{2}, 1),$ if $M^s:=\max\limits_{\xi\in\bar{B}_1}h^sV^se^{-u^{s*2}}$ is achieved in $B_1,$ then
$M^s\leq C,$ where $C=C(|u^{s*}|_{C^0})$ is a constant independent of $s.$
\end{lemm}
\begin{proof}
For our convenience, in this proof, we drop the superscript $s$ from $u^{s*},$ $h^s,$ and $V^s.$
Consider $\phi=hVe^{-u^{*2}}$ and assume $\phi$ achieves its maximum at an interior point $\xi_0\in B_1.$ We may rotate the coordinate such that at $\xi_0,$
we have $V=u^*_1$. Differentiating $\phi$ we get
\[\frac{2s\xi_i}{h}=\frac{u_1^*u^*_{1i}}{V}-2u^*u^*_i,\,\,1\leq i\leq n.\]
Therefore, when $i=1,$ we obtain
\[\frac{2s\xi_1}{h}=u^*_{11}-2u^*V.\]
By the convexity of $u^*$ we know that $u^*_{11}>0,$ which gives
\[\frac{2s\xi_1}{h}\geq 2|u^*|V.\]
This completes the proof of the Lemma \ref{loc-c1-lem1}.
\end{proof}
Finally, we want to show $h^sV^s$ is bounded on $\p B_1.$
\begin{lemm}
\label{loc-c1-lem2}
For any $s\in [\frac{1}{2}, 1),$ if $M^s:=\max\limits_{\xi\in\bar{B}_1}h^sV^se^{-u^{s*2}}$ is achieved in $\p B_1,$ then
$M^s\leq C,$ where $C>0$ is a constant independent of $s.$
\end{lemm}
\begin{proof}
For our convenience, in this proof, we drop the superscript $s$ from $u^{s*},$ $h^s,$ and $V^s.$
Let's $\ubar{\psi}=-kh^{1/2}+k(1-s)^{1/2}+u_0^*.$ For any $\xi\in B_1,$ WLOG, we may assume $\xi=(r, 0, \cdots, 0).$
A direct calculation yields at $\xi$ we have
\[\ubar{\psi}_{11}=ksh^{-3/2}+(u^*_0)_{11},\]
\[\ubar{\psi}_{ii}=ksh^{-1/2}+(u^*_0)_{ii}\,\,\mbox{for $i\geq 2$,}\]
and
\[\ubar{\psi}_{ij}=(u^*_0)_{ij}\,\,\mbox{for all other cases.}\]
Since $u^*_0$ is strictly convex we get
\[\s_n(D^2\ubar{\psi})>\s_n(D^2(-kh^{1/2}))=(ks)^nh^{-\frac{n+2}{2}}.\]
Choosing $k=\frac{3}{[\min(-\varphi^*)]^{\al/n}},$ we can see that $\ubar{\psi}$ is a subsolution of \eqref{gcf-approx}. Thus, on $\p B_1$ we have
\be\label{gcf2*}
|Du^*|<|D\ubar{\psi}|<\frac{C}{\sqrt{1-s}},
\ee
where $C>0$ is a constant independent of $s.$ It is easy to see that \eqref{gcf2*} implies the Lemma.
\end{proof}

Combining Lemma \ref{loc-c1-lem1} and Lemma \ref{loc-c1-lem2} we conclude
\begin{lemm}
\label{lem-5.6}
Let $u^{s*}$ be the solution of \eqref{gcf-approx} for $s\in[\frac{1}{2}, 1).$ Then there exists a constant $C>0,$
such that $|Du^{s*}(\xi)|(1-s|\xi|^2)\leq C.$ Here, $C$ is a constant independent of $s.$
\end{lemm}

\subsubsection{Local $C^1$ lower bound near $\p B_1$} In this part, we will show $|Du^{s^*}|(\xi)\goto\infty$ as $s, |\xi|\goto 1.$

In order to obtain local $C^1$ lower bounds, we will construct a supersolution $\bar{u}^{s*}$ to
\be\label{lce0.1}
\s_n(D^2u^*)=\frac{(1-s|\xi|^2)^{(\al-n-2)/2}}{(-u^*)^\al}\,\,\mbox{in $B_1$}\ee for $\al\leq n,$
which satisfies
\be\label{lce0.2}
|D\bar{u}^{s*}(\xi)|\goto\infty \,\,\mbox{as $s\goto 1$ and $|\xi|\goto 1.$}
\ee
In the following, we will restrict ourselves to the case when $s\in[1/2, 1)$.
Denote $h^s=1-s|\xi|^2,$ then $h^s_i=-2s\xi_i,$ and $h^s_{ij}=-2s\delta_{ij}.$ Consider $g_1(h^s)=-h^s\log|\log h^s|.$ By a straightforward calculation we get
\be\label{gcf5}
g_1'=-\log|\log h^s|+\frac{1}{|\log h^s|},
\ee
and
\be\label{gcf6}
g_1''=\lt(1+\frac{1}{|\log h^s|}\rt)\frac{1}{h^s|\log h^s|}.
\ee
Therefore, at any point $\xi\in B_1$ with $|\xi|=r$ we have
\be\label{gcf7}
\begin{aligned}
\det(D^2g_1)&=s^n\lt(2\log|\log h^s|-\frac{2}{|\log h^s|}\rt)^{n-1}\\
&\times\lt[\lt(2\log|\log h^s|-\frac{2}{|\log h^s|}\rt)+4sr^2\lt(1+\frac{1}{|\log h^s|}\rt)\frac{1}{h^s|\log h^s|}\rt].\\
\end{aligned}
\ee
When $h^s<\delta_0$ for some fixed $\delta_0>0$ small, we have $\det(D^2g_1)\leq \frac{C}{h^s}$ for some constant $C>0.$ Here, $C=C(\delta_0)$ is independent of $s.$
On the other hand, when $h^s\geq\delta_0,$ it's easy to see that
$g_2=\frac{|\xi|^2}{2}-\frac{1-\delta_0}{2s}-\delta_0\log|\log\delta_0|$ satisfying
\be\label{gfc8}
\left\{
\begin{aligned}
\det(D^2g_2)&=1\,\,&\mbox{in $B_{\sqrt{\frac{1-\delta_0}{s}}}$}\\
g_2&=-\delta_0\log|\log\delta_0|\,\,&\mbox{on $\p B_{\sqrt{\frac{1-\delta_0}{s}}}$ }.
\end{aligned}
\right.
\ee
Define
\[
g=\left\{\begin{aligned}&g_1\,\,\mbox{ for $h^s<\delta_0,$}\\
&g_2\,\,\mbox{for $\delta_0\leq h^s\leq 1,$}
\end{aligned}\right.
\]
then $g$ is a continuous and convex function in $B_1.$ By standard smoothing procedure, we can find a convex, rotationally symmetric function $\Phi\in C^2(B_1)$ such that
\[
\Phi (g)=\left\{\begin{aligned}
&g_1\,\,&\mbox{for $h^s<\frac{\delta_0}{2}$,}\\
&g_2\,\, &\mbox{for $2\delta_0\leq h^s\leq1$.}
\end{aligned}
\right.
\]
We can see that for some suitable choice of $\rho>0,$ $\rho\Phi$ is a supersolution of \eqref{lce0.1} that satisfies \eqref{lce0.2}. Here, $\rho>0$ only depends on $|u^{s*}|_{C^0}.$ Below we will denote this supersolution by $\bar{u}^{s*}.$

Following \cite{AML} we can prove following Lemmas
\begin{lemm}
\label{c1-lem}
Let $u^{s*}$ be the solution of \eqref{gcf-approx}, then
$|\p u^{s*}|$ is bounded above by a constant $C_1=C_1(|\varphi^*|_{C^1}).$
\end{lemm}
\begin{proof}
For our convenience, in this proof, we drop the superscript $s$ from $u^{s*}.$
We take the logarithms of both sides of \eqref{gcf-approx} and differentiate it with respect to $\xi_k,$ then find
\[u^{*ij}u^*_{kij}=\frac{\p}{\p\xi_k}\lt[\log(1-s|\xi|^2)\cdot\frac{\al-n-2}{2}\rt]-\al\frac{\p\log(-u^*)}{\p\xi_k}.\]
This implies
\be
\label{gcf9}
\sum u^{*ij}(\p u^*)_{ij}=-\al\frac{\p(-u^*)}{-u^*}=\al\frac{\p u^*}{-u^*}.
\ee
If $\p u^*$ achieves interior positive maximum, we would have $0\geq \al\frac{\p u^*}{-u^*}>0.$ This leads to a contradiction.
Similarly, if $\p u^*$ achieves interior negative minimum, we would have $0\leq \al\frac{\p u^*}{-u^*}<0.$ This also leads to a contradiction.
Therefore we conclude
\[|\p u^*|\leq\max\limits_{\p B_1}|\p\varphi^*|.\]
\end{proof}

\begin{lemm}
\label{c2-upper-lem}
Let $u^{s*}$ be the solution of \eqref{gcf-approx}, then
$\p^2 u^{s*}$ is bounded above by a constant $C_2=C_2(|\varphi^*|_{C^2}).$
\end{lemm}
\begin{proof}
For our convenience, in this proof, we drop the superscript $s$ from $u^{s*}.$
We have shown
\[\sum u^{*ij}(\p u^*)_{ij}=-\al\frac{\p(-u^*)}{-u^*}=\al\frac{\p u^*}{-u^*}.\]
Differentiating this equation once again we obtain
\be\label{gcf10}
\sum u^{*ij}\p[(\p u^*)_{ij}]+\p(u^{*ij})(\p u^*)_{ij}
=\frac{\al}{(- u^*)^2}(\p u^*)^2+\frac{\al}{(-u^*)}\p^2 u^*.
\ee
Following the argument of Lemma 5.2 in \cite{AML}, we get
\be\label{gcf11}
\sum u^{*ij}[(\p^2 u^*)_{ij}]\geq\frac{\al}{(-u^*)}(\p^2 u^*).
\ee
Therefore, $\p^2 u^*$ does not achieve positive maximum at interior points and we conclude
\[\p^2 u^*\leq\max\limits_{\p B_1}|\p^2\varphi^*|.\]
\end{proof}

\begin{lemm}
\label{local-upper-barrier-lem}
Let $s\in [1/2, 1),$ $\sqrt{\frac{2-\delta_0}{2s}}<r<1,$ and $\dS^{n-1}=\{\xi\in\R^n\mid\sum\xi_i^2=r^2\}.$ For any any point
$\hat{\xi}\in\dS^{n-1}(r)$ there exists a function
\[\uus_s=\bar{u}^{s*}+b_1\xi_1+\cdots+b_n\xi_n+d\]
such that $\uus_s(\hat{\xi})=u^{s*}(\hat{\xi})$ and $\uus_s(\xi)>u^{s*}(\xi)$ for any $\xi\in\dS^{n-1}\setminus\{\hat{\xi}\}.$ Here, $u^{s*}$ is a solution of \eqref{gcf-approx}, $\bar{u}^{s*}$ is the supersolution constructed before,
$b_1, \cdots, b_n$ are constants depending on $\hat{\xi},$ and $d$ is a positive constant independent of $\hat{\xi}$ and $r.$
\end{lemm}
\begin{proof}
By rotating the coordinate we may assume $\hat{\xi}=(r, 0,\cdots, 0).$ We choose $b_k=\frac{\p u^{s*}}{\p\xi_k}(r, 0, \cdots, 0),$ $k=2, 3, \cdots, n,$ and
choose $b_1$ such that $u^{s*}(r, 0, \cdots, 0)=\bar{u}^{s*}(r, 0, \cdots, 0)+b_1r+d.$ To choose $d$ we consider an arbitrary great circle $c(t)$ on $\dS^{n-1}(r)$ passing through $\hat{\xi},$
for example the circle
\[\xi_1=r\cos t,\,\, \xi_2=r\sin t, \,\, -\pi\leq t\leq\pi,\,\,\xi_3=\xi_4=\cdots=\xi_n=0.\]
Let
\[
\begin{aligned}
F(t)=(\uus_s-u^{s*})|_{c(t)}&=\bar{u}^{s*}|_{c(t)}+b_1\xi_1+\cdots+b_n\xi_n+d-u^{s*}|_{c(t)}\\
&=\bar{u}^{s*}|_{c(t)}+b_1r\cos t+b_2r\sin t+d-u^{s*}|_{c(t)}.
\end{aligned}
\]
Note that by our construction of $\bar{u}^{s*},$ when $\frac{2-\delta_0}{2s}<r<1$ we have
\[\bar{u}^{s*}|_{c(t)}=-\rho h^s\log|\log h^s||_{\{|\xi|=r\}}:=\bar{u}^{s*}(r).\]
Therefore, we get
\[F(t)=\bar{u}^{s*}(r)+[u^{s*}(r, 0,\cdots, 0)-\bar{u}^{s*}(r)-d]\cos t+b_2r\sin t+d-u^{s*}(t).\]
It's clear that $F(0)=0$ and $\frac{dF}{dt}(0)=0.$ We will look at the second derivative of $F$.
Since \[\frac{d^2 F(t)}{dt^2}=[d+\bar{u}^{s*}(r)-u^{s*}(r, 0, \cdots, 0)]\cos t-b_2r\sin t-\frac{d^2u^{s*}}{dt^2},\]
when $-\frac{\pi}{4}\leq t\leq \frac{\pi}{4}$ we choose $d>u^{s*}(r, 0, \cdots, 0)-\bar{u}^{s*}(r)$ then we get
\[
\begin{aligned}
\frac{d^2F(t)}{dt^2}&\geq\frac{1}{\sqrt{2}}[d+\bar{u}^{s*}(r)-u^{s*}(r, 0, \cdots, 0)]-\lt|\frac{du^{s*}}{dt}(0)\rt|-\frac{d^2u^{s*}}{dt^2}\\
&\geq \frac{1}{\sqrt{2}}[d+\bar{u}^{s*}(r)-u^{s*}(r, 0, \cdots, 0)]-C_3
\end{aligned}
\]
for some $C_3>0$ determined by Lemma \ref{c1-lem} and \ref{c2-upper-lem}.
When $t\in[-\pi, -\frac{\pi}{4}]\cup[\frac{\pi}{4}, \pi]$ we have
\[
\begin{aligned}
F(t)&=d(1-\cos t)+[u^{s*}(r,0\cdots, 0)-\bar{u}^{s*}(r)]\cos t+b_2r\sin t-u^{s*}(t)+\bar{u}^{s*}(r)\\
&\geq d\lt(1-\frac{\sqrt{2}}{2}\rt)-C_4
\end{aligned}
\]
By choosing $d>0$ sufficiently large we prove this lemma.
\end{proof}

Finally, we can prove
\begin{lemm}
\label{lem-5.6*}
Let $u^{s*}$ be the solution of \eqref{gcf-approx} for $s\in[1/2, 1)$. Then there exists $C=C(|\varphi^*|_{C^2})>0,$ such that when
$\sqrt{\frac{2-\delta_0}{2s}}\leq\sqrt{\frac{2-\delta_1}{2s}}<|\xi|<1,$ we have
\[\frac{|Du^{s*}(\xi)|}{\log|\log h^s|}\geq C.\]
Here, $\delta_1>0$ is a small constant.
\end{lemm}
\begin{proof}
When $\sqrt{\frac{2-\delta_0}{2s}}\leq\sqrt{\frac{2-\delta_1}{2s}}< r<1,$ for any $\hat{\xi}\in\dS^{n-1}(r)$ we assume $\hat{\xi}=(r, 0, \cdots, 0).$
By Lemma \ref{local-upper-barrier-lem}, there exists a supersolution of \eqref{gcf-approx} $\uus_s,$ such that $\uus_s(\hat{\xi})=u^{s*}(\hat{\xi})$ and
$\uus_s(\xi)>u^{s*}(\xi)$
for any $\xi\in\dS^{n-1}\setminus\{\hat{\xi}\}.$ By the maximum principle we get $\uus_s(\xi)>u^{s*}(\xi)$ in $B_r.$
Hence at $\hat{\xi}$ we obtain
\[\frac{\p u^{s*}}{\p\xi_1}>\frac{\p\uus_s}{\p\xi_1}=\frac{\p\bar{u}^{s*}}{\p\xi_1}+b_1.\]
Therefore, when $\delta_1>0$ is chosen to be small we complete the proof of this Lemma.
\end{proof}

\subsection{Local $C^2$ estimates}
Lemma \ref{lem-5.6} gives us local $C^1$ estimates for $u^{s*}$. In the following we will establish local $C^2$ estimates for the solution
$u^{s*}$ of equation \eqref{gcf-approx}. Comparing with usual local $C^2$ estimates, the complication here is as $s\goto 1$ and $|\xi|\goto 1,$ by Lemma \ref{lem-5.6*}
we know that $|Du^{s*}(\xi)|\goto\infty.$ In other words, we don't have uniform $C^1$ estimates. Therefore, we need to introduce some new techniques to overcome this difficulty.

Let $u_0^*$ be the solution of \eqref{gcf1.1}, denote $\eta^s:=u_0^*-u^{s*}$ and $f^s=(-u^{s*})^{-\al}(1-s|\xi|^2)^{(\al-n-2)/2},$ we prove
\begin{lemm}
\label{c2-local-aux}
Let $u^{s*}$ be a solution of \eqref{gcf-approx} for $s\in [1/2, 1).$ Then we have
\be\label{gcf-compare}
\eta^s<C(h^s)^{m_\al},
\ee
where $m_\al:=\frac{n-2+\al}{2n},$ $h^s=1-s|\xi|^2,$ and $C=C(n, \al, |u^*|_{C^0})>0$ is a constant independent of $s.$
\end{lemm}
\begin{proof}
For our convenience, in this proof, we will drop the superscript $s$ on $\eta^s, h^s, f^s,$ and $u^{s*}.$
Let $\gamma=\frac{1}{m_\al},$ since $\al\in(0, n],$ it's clear that $\gamma>1.$
Assume $\max\limits_{\xi\in B_1}\eta^\gamma h^{-1}$ is achieved at an interior point $\xi_0.$ We may rotate the coordinate such that at this point
$u^*_{ij}=u^*_{ii}\delta_{ij}.$ Moreover, at $\xi_0$ we have
\[0=\gamma\frac{\eta_i}{\eta}-\frac{h_i}{h},\]
and
\[
\begin{aligned}
0&\geq\gamma\s_n^{ii}\lt(\frac{\eta_{ii}}{\eta}-\frac{\eta_i^2}{\eta^2}\rt)-\frac{\s_n^{ii}h_{ii}}{h}+\s_n^{ii}\frac{h_i^2}{h^2}\\
&=\gamma\s_n^{ii}\frac{\eta_{ii}}{\eta}+(\gamma^2-\gamma)\s_n^{ii}\frac{\eta_i^2}{\eta^2}+2s\frac{\sum\s_n^{ii}}{h}.
\end{aligned}
\]
Since $u_0^*$ is convex, we get $\s_n^{ii}(u_0^*)_{ii}>0,$ the above inequality becomes
\[0\geq\frac{-n\gamma\s_n}{\eta}+2s\frac{\sum\s_n^{ii}}{h}.\]
Recall that $\sum\s_n^{ii}=\s_{n-1}\geq c(n)\s_n^{\frac{n-1}{n}}$ and $s\in[1/2, 1),$ we conclude
\[n\gamma\geq\frac{c(n)\eta}{h\s_n^{1/n}}\geq c_0\frac{\eta}{h^{1-\frac{n+2-\al}{2n}}}=c_0\frac{\eta}{h^{m_\al}},\]
where $c(n)>0$ is a constant depending on $n$ and $c_0$ is a constant depending on $|u^*|_{C^0}$ and $\al.$
Therefore, we conclude that at $\xi_0$
\[\eta^\gamma h^{-1}\leq C,\]
where $C=C(n, \al, |u^*|_{C^0})$ is independent of $s$.
\end{proof}

\begin{lemm}
\label{gcf-local-c2-lem}
Let $u^{s*}$ be a solution of \eqref{gcf-approx} for $s\in[1/2, 1)$. Then we have
\[\max\limits_{\xi\in B_1,\,\, \zeta\in\mathbb{S}^n}\eta^{\beta}u^{s*}_{\zeta\zeta}\leq C.\]
Here, $\beta=\frac{8}{m_\al}$ and $C$ only depends on the $C^0$ estimates of $u^{s*}$ and the local $C^1$ estimates we obtained in Lemma \ref{lem-5.6}.
\end{lemm}
\begin{proof}
In this proof, for our convenience, we will drop the superscript $s$. We denote
$h=1-s|\xi|^2,$ then $h_i=-2s\xi_i$ and $h_{ij}=-2s\delta_{ij}.$
We also note, differentiating $f=(-u^*)^{-\al}h^{\frac{\al-n-2}{2}}$ twice we get
\[f_i=f\lt[\frac{\al u^*_i}{-u^*}+\frac{(\al-n-2)}{2}h^{-1}h_i\rt]\]
and
\[
\begin{aligned}
f_{ii}&=f\lt[\frac{\al u^*_i}{-u^*}+\frac{(\al-n-2)}{2}h^{-1}h_i\rt]^2\\
&+f\lt[\frac{\al u^*_{ii}}{-u^*}+\frac{\al u_i^{*2}}{u^{*2}}-\frac{(\al-n-2)}{2}h^{-2}h_i^2+\frac{(\al-n-2)}{2}h^{-1}h_{ii}\rt].
\end{aligned}
\]
Moreover, applying Lemma \ref{lem-5.6} we may assume
\[h^2|Du^*|^2<\m_0\,\,\mbox{and $h^4|Du^*|^2<\m_0,$}\]
for some positive constant $\m_0>1.$
Let $g=h^4|Du^*|^2$ and differentiate $g$ twice, we get
\be\label{gcf1.2}
g_i=4h^3h_i|Du^*|^2+2h^4\sum_ku_k^*u^*_{ki},
\ee
and
\be\label{gcf1.3}
\begin{aligned}
g_{ii}&=12h^2h^2_i|Du^*|^2+4h^3|Du^*|^2h_{ii}+16h^3\sum_kh_iu_k^*u_{ki}^*\\
&+2h^4\sum_{k}u^{*2}_{ki}+2h^4\sum_ku^*_ku^*_{kii}.
\end{aligned}
\ee
Now we consider $\phi=\frac{\eta^\beta u^*_{\zeta\zeta}}{1-\frac{g}{M}},$ where $\beta>0,$ $M>2\m_0$ are some constants to be determined, and $\zeta\in\mathbb{S}^n$ is some direction. Suppose
\[\hat{M}:=\max\limits_{\xi\in B_1,\,\, \zeta\in\mathbb{S}^n}\phi\] is achieved at an interior point $\xi_0\in B_1$ in the direction of $\zeta_0\in\mathbb{S}^n.$
We may choose a local orthonormal frame $\{e_1, \cdots, e_n\}$ at $\xi_0,$ such that $u^*_{ij}(\xi_0)$ is diagonal and we also assume $\zeta_0=e_1.$
Then at $\xi_0$ we have
\[\log\phi=\beta\log\eta-\log\lt(1-\frac{g}{M}\rt)+\log u_{11}^*.\]
Differentiating $\log\phi$ twice we get
\be\label{gcf1.4}
0=\frac{\phi_i}{\phi}=\frac{\beta\eta_i}{\eta}+\frac{g_i}{M-g}+\frac{u_{11i}^*}{u_{11}^*}
\ee
and
\be\label{gcf1.5}
0\geq\s_n^{ii}\lt[\frac{\beta\eta_{ii}}{\eta}-\frac{\beta\eta_i^2}{\eta^2}+\frac{g_{ii}}{M-g}+\frac{g_i^2}{(M-g)^2}+\frac{u^*_{11ii}}{u^*_{11}}-
\lt(\frac{u^*_{11i}}{u^*_{11}}\rt)^2\rt].
\ee
By \eqref{gcf1.4} we can see that when $i=1$ we have
\be\label{number1}
\lt(\frac{u^*_{111}}{u^*_{11}}\rt)^2=\lt(\frac{\beta\eta_1}{\eta}+\frac{g_1}{M-g}\rt)^2\leq\frac{2\beta^2\eta_1^2}{\eta^2}+\frac{2g_1^2}{(M-g)^2}.
\ee
When $i\geq 2$
\be\label{number2}
\beta\lt(\frac{\eta_i}{\eta}\rt)^2=\frac{1}{\beta}\lt(\frac{g_i}{M-g}+\frac{u^*_{11i}}{u^*_{11}}\rt)^2
\leq\frac{2g_i^2}{\beta(M-g)^2}+\frac{2}{\beta}\lt(\frac{u^*_{11i}}{u^*_{11}}\rt)^2.
\ee
Note also that
\[\sum_k\s_n^{ii}u_k^*u^*_{kii}=\sum_ku_k^*f_k=f\lt(\frac{\al|Du^*|^2}{-u^*}+\frac{(\al-n-2)}{2}h^{-1}\sum_kh_ku^*_k\rt).\]
Therefore,
\be\label{gcf1.6}
\begin{aligned}
\s_n^{ii}g_{ii}&\geq 12h^2|Du^*|^2\s_n^{ii}\xi_i^2-8h\m_0\sum\s_n^{ii}-32h^2\sqrt{\m_0}n\s_n\\
&+2h^4\s_n\s_1+2h^4\s_n\lt(\frac{\al|Du^*|^2}{-u^*}+\frac{(\al-n-2)}{2}h^{-1}\sum_kh_ku^*_k\rt),
\end{aligned}
\ee
and
\be\label{gcf1.7}
\s_n^{11}g_1^2\leq\s_n^{11}\lt(32h^6h_1^2|Du^*|^4+8h^8u_1^{*2}u^{*2}_{11}\rt)
<128\m_0h^4|Du^*|^2\s_n^{11}\xi_1^2+8h^6\m_0\s_nu^*_{11}.
\ee
Combining \eqref{gcf1.6} and \eqref{gcf1.7} we obtain
\be\label{gcf1.8}
\begin{aligned}
&\frac{\s_n^{ii}g_{ii}}{M-g}-\frac{\s_n^{11}g_1^2}{(M-g)^2}\\
&\geq\frac{1}{(M-g)^2}\big\{(M-g)[12h^2|Du^*|^2\s_n^{ii}\xi_i^2-36h^2\sqrt{\m_0}n\s_n-8h\m_0\sum\s_n^{ii}+2h^4\s_n\s_1]\\
&-128\m_0h^4|Du^*|^2\s_n^{11}\xi_1^2-8h^6\m_0\s_nu^*_{11}\big\}\\
\end{aligned}
\ee
Choose $M=13\m_0+N$ such that $M-g\geq 12\m_0+N$ then
\be\label{gcf1.9}
\begin{aligned}
&\frac{\s_n^{ii}g_{ii}}{M-g}-\frac{\s_n^{11}g_1^2}{(M-g)^2}\\
&\geq\frac{1}{(M-g)^2}[-(M-g)36h^2\sqrt{\m_0}n\s_n-8(M-g)h\m_0\sum\s_n^{ii}+2Nh^4\s_n\s_1],
\end{aligned}
\ee
where we have used $\s_1>u^*_{11}.$
Differentiating $\s_n=f$ twice we get
\[\s_n^{ii}u^*_{11ii}+\s_n^{pq, rs}u^*_{pq1}u^*_{rs1}=f_{11}.\]
Thus,
\be\label{gcf1.10}
\begin{aligned}
\s_n^{ii}u^*_{11ii}&=f_{11}+\sum\limits_{p\neq q}\s_n^{pp, qq}u^{*2}_{pq1}-\sum\limits_{p\neq q}
\s_n^{pp, qq}u^*_{pp1}u^*_{qq1}\\
&\geq f_{11}+2\sum_{p=2}^n\frac{\s_n}{u^*_{pp}u^*_{11}}u^{*2}_{11p}-\frac{f^2_1}{f}.
\end{aligned}
\ee
Notice that
\[
\begin{aligned}
&f_{11}-\frac{f_1^2}{f}\\
&=f\lt[\frac{\al u^*_{11}}{-u^*}+\frac{\al u_1^{*2}}{-u^*}+\frac{(n+2-\al)}{2}h^{-2}h_1^2+(n+2-\al)sh^{-1}\rt]\\
&\geq C_5 u^*_{11}f,
\end{aligned}
\]
we conclude
\be\label{gcf1.11}
\frac{\s_n^{ii}u^*_{11ii}}{u_{11}^*}\geq C_5f+2\sum_{p=2}^n\frac{\s_n^{pp}}{u^{*2}_{11}}u^{*2}_{11p}.
\ee
By a straightforward calculation we can see
\[\s_n^{ii}\eta_{ii}=\s_n^{ii}((u_{0}^*)_{ii}-u^*_{ii})\geq C_6\sum\s_n^{ii}-n\s_n.\]
Combining \eqref{number1}, \eqref{number2} with \eqref{gcf1.5} we obtain
\[
\begin{aligned}
0&\geq\s_n^{ii}\frac{\beta\eta_{ii}}{\eta}-\frac{(\beta+2\beta^2)\s_n^{11}\eta_1^2}{\eta^2}+\frac{\s_n^{ii}g_{ii}}{M-g}-\frac{\s_n^{11}g_1^2}{(M-g)^2}\\
&-\lt(1+\frac{2}{\beta}\rt)\sum\limits_{i\geq 2}\s_n^{ii}\lt(\frac{u^*_{11i}}{u^*_{11}}\rt)^2+\frac{\s_n^{ii}u^*_{11ii}}{u^*_{11}}
+\lt(1-\frac{2}{\beta}\rt)\sum_{i\geq 2}\frac{\s_n^{ii}g_i^2}{(M-g)^2}.
\end{aligned}
\]
When $\beta\geq 2,$ applying \eqref{gcf1.9} and\eqref{gcf1.11} we get
\be\label{gcf1.12}
\begin{aligned}
0&\geq\frac{\beta}{\eta}(C_6\s_n^{ii}-n\s_n)-\frac{(\beta+2\beta^2)\s_n((u^*_{0})_1-u^*_1)^2}{\eta^2}\\
&-\frac{36h^2n\sqrt{\m_0}\s_n}{N}-\frac{8\m_0h\sum\s_n^{ii}}{N}+\frac{2Nh^4\s_n\s_1}{M^2}+C_5\s_n
\end{aligned}
\ee
We will choose $N$ large such that $\frac{C_6\beta}{\eta}>\frac{8\m_0}{N}.$
Therefore, \eqref{gcf1.12} becomes
\be\label{gcf1.12*}0\geq-\frac{n\beta}{\eta}\s_n-\frac{(\beta+2\beta^2)\s_n(C_6+|Du^*|)^2}{\eta^2}-C_7\s_n+\frac{2Nh^4\s_n u^*_{11}}{M^2}.\ee
By Lemma \ref{c2-local-aux} we know $\eta^\frac{1}{m_\al}<Ch,$ which gives $h>C\eta^{\frac{1}{m_\al}}.$ Thus, we have
$$\eta^{\frac{1}{m_\al}}|Du^*|<C\m_0.$$

Now, let $\beta=\frac{8}{m_\al}>8$ and multiplying \eqref{gcf1.12*} by $\eta^{\frac{\beta}{2}},$ we obtain
\[0\geq-n\beta\eta^{\frac{\beta}{2}-1}-(\beta+2\beta^2)(C_6+|Du^*|)^2\eta^{\frac{\beta}{2}-2}-C_7\eta^{\frac{\beta}{2}}
+C_8N\frac{\eta^\beta u^*_{11}}{M^2}.\]
Therefore, we conclude that $\eta^{\beta}u^*_{11}<C_9$ at its interior maximum point, which implies $\phi<2C_9.$
\end{proof}

\begin{proof}[Proof of Theorem \ref{int-thm1}]
By subsection 2.1 we know there exists a solution $u^{s*}$ of \eqref{gcf-approx} for any $s\in (0, 1)$. Combining Lemma \ref{c0-gcf-approx-lem}, \ref{lem-5.6}, \ref{gcf-local-c2-lem} with the classic regularity theorem, we know that there exists a subsequence of $u^{s*}$ denoted by $\{u^{s_j*}\}_{j=1}^\infty,$ converging locally
smoothly to a convex function $u^*,$ which satisfies \eqref{gcf2}. Here, $s_j\goto 1$ as $j\goto\infty.$
Moreover, applying Lemma \ref{lem-5.6*} and Lemma 14 of \cite{WX20} we conclude, the Legendre transform of $u^*,$ denoted by $u,$ is the desired entire solution of \eqref{gcf1} satisfying the asymptotic condition \eqref{gcf1*}. This completes the proof of Theorem \ref{int-thm1}.
\end{proof}

\section{$\s_k$ curvature self-expander}
\label{skc}
In this section we will show that there exists an entire, strictly spacelike solution to the following equation
\be\label{se1}
\sigma_k(\ka[\M_u])=\lt(-\lt<X, \nu\rt>\rt)^\al,
\ee
and
\be\label{se1*}
u(x)-|x|\goto\varphi\lt(\frac{x}{|x|}\rt)\,\,\mbox{as $|x|\goto\infty,$}
\ee
where $0<\al\leq k$ are constants.
If $u$ is a strictly convex solution satisfying \eqref{se1} and \eqref{se1*}, then subsection 2.3 and Lemma 14 of \cite{WX20} imply its Legendre transform $u^*$ satisfies
\be\label{se2}
\left\{
\begin{aligned}
\frac{\sigma_n}{\s_{n-k}}(w^*\ga_{ik}^*u^*_{kl}\ga_{lj}^*)&=\lt(\frac{\sqrt{1-|\xi|^2}}{-u^*}\rt)^\al\,\,&\mbox{in $B_1$}\\
u^*&=\varphi^*\,\,&\mbox{on $\p B_1$},
\end{aligned}
\right.
\ee
here $\varphi^*(\xi)=-\varphi(\xi).$
By Section \ref{gcf} we know there exists $\ubar{u}$ such that
\[\s_n(\ka[\M_{\ubar{u}}])=\frac{1}{{n\choose k}^{\frac{n}{k}}}\lt(-\lt<X, \nu\rt>\rt)^{\frac{\al n}{k}},\]
and $\ubar{u}(x)-|x|\goto\varphi\lt(\frac{x}{|x|}\rt)$ as $|x|\goto\infty.$
Applying Maclaurin's inequality we obtain
\[\s_k(\ka[\M_{\ubar{u}}])\geq\lt(-\lt<X, \nu\rt>\rt)^\al.\]
We will denote the Lengendre transform of $\ubar{u}$ by $\lus,$ then $\lus$ satisfies
\be\label{skc0.1}
\left\{\begin{aligned}
\frac{\s_n}{\s_{n-k}}(w^*\ga^*_{ik}\lus_{kl}\ga^*_{lj})&\leq\lt(\frac{\sqrt{1-|\xi|^2}}{-\lus}\rt)^\al\,\,&\mbox{in $B_1$},\\
\lus&=\varphi^*\,\,&\mbox{on $\p B_1$.}
\end{aligned}
\right.
\ee
We will study the following approximate equation
\be\label{skc0.2}
\left\{\begin{aligned}
\frac{\s_n}{\s_{n-k}}(w^*\ga^*_{ik}u^*_{kl}\ga^*_{lj})&=\lt(\frac{\sqrt{1-|\xi|^2}}{-u^*}\rt)^\al\,\,&\mbox{in $B_r$},\\
u^*&=\lus\,\,&\mbox{on $\p B_r$,}
\end{aligned}
\right.
\ee
where $0<r<1.$
In the following we denote $\Psi^*:=\lt(\frac{\sqrt{1-|\xi|^2}}{-\lus}\rt)^\al,$ and we can see that as long as $-u^*>0$ we have
\[\frac{\p\Psi^*}{\p u^*}=\al(\sqrt{1-|\xi|^2})^\al(-u^*)^{-\al-1}>0.\]
This guarantees that the maximum principle holds for \eqref{skc0.2}.
Now assume $\max\limits_{\xi\in\p B_1}\varphi^*(\xi)=-C_0<0,$ let $\bar{u}$ be a constant $\s_k$ curvature hypersurface satisfying
$\s_k(\ka[\M_{\bar{u}}])=C_0^\al,$ $\bar{u}$ is strictly convex, and $\bar{u}(x)-|x|\goto\varphi\lt(\frac{x}{|x|}\rt)$
as $|x|\goto\infty.$
We denote the Legendre transform of $\bar{u}$ by $\uus,$ then $\uus$ satisfies
\be\label{skc0.3}
\left\{
\begin{aligned}
\frac{\s_n}{\s_{n-k}}(w^*\ga^*_{ik}\uus_{kl}\ga^*_{lj})&=\frac{1}{C_0^\al}\geq\lt(\frac{\sqrt{1-|\xi|^2}}{-\lus}\rt)^\al
\,\,&\mbox{ in $B_1$}\\
\uus&=\varphi^*\,\,&\mbox{on $\p B_1.$}
\end{aligned}
\right.
\ee
By the maximal principle we know $\uus<\lus$ in $B_1.$ Moreover, for any solution $u^{r*}$ of \eqref{skc0.2}, it is easy to see that
\[\uus<u^{r*}<\lus\,\,\mbox{in $B_r.$}\]
Therefore, we conclude
\begin{lemm}
\label{c0-sk-lemma}
Let $u^{r*}$ be a solution of \eqref{skc0.2} and $\lus, \uus$ are constructed above. Then we have
\[\uus<u^{r*}<\lus\,\,\mbox{in $B_r.$}\]
\end{lemm}

\subsection{Global a priori estimates}
In the subsection, we will prove a priori estimates that needed for the solvability of \eqref{skc0.2}.
\label{ape}
\begin{lemm}
\label{c1-sk-lemma}
Let $u^{r*}$ be a solution of \eqref{skc0.2}, then there exists $C>0$ such that
\[|Du^{r*}|<C.\]
\end{lemm}
\begin{proof}
By Section 2 of \cite{CNS3}, we know that for any $0<r<1,$ we can construct a subsolution $\ubar{u}^{r*}$ such that
\[
\begin{aligned}
\frac{\s_n}{\s_{n-k}}(w^*\ga_{ik}^*\ubar{u}_{kl}^{r*}\ga^*_{lj})&\geq\frac{1}{C_0^{\al}}\,\,&\mbox{in $B_r$}\\
\ubar{u}^{r*}&=\ubar{u}^*\,\,&\mbox{on $\p B_r.$}
\end{aligned}
\]
Then by the convexity of $u^{r*}$ we have
\[|Du^{r*}|\leq\max\limits_{\p B_r}|D\ubar{u}^{r*}|.\]
\end{proof}

Let $v=\lt<X, \nu\rt>=\frac{x\cdot Du-u}{\sqrt{1-|Du|^2}}=\frac{u^*}{\sqrt{1-|Du|^2}}.$ We will consider the hyperbolic model of \eqref{skc0.2} (seeing \cite{WX20} for detail).
\be\label{ape0.1}
\left\{
\begin{aligned}
F(v_{ij}-v\delta_{ij})&=(-v)^{-\al}\,\,&\mbox{in $U_r,$}\\
v&=\frac{\lus}{\sqrt{1-r^2}}\,\,&\mbox{on $\p U_r,$}
\end{aligned}
\right.
\ee
where $v_{ij}$ denotes the covariant derivative with respect to the hyperbolic metric, $U_r =P^{-1}(B_r)\subset\mathbb{H}^n(-1),$
$F(v_{ij}-v\delta_{ij})=\frac{\s_n}{\s_{n-k}}(\lambda[v_{ij}-v\delta_{ij}]),$ and $\lambda[v_{ij}-v\delta_{ij}]=(\lambda_1, \cdots, \lambda_n)$ denotes
the eigenvalues of $(v_{ij}-v\delta_{ij}).$
Recall the following Lemma 27 from \cite{RWX19}.
\begin{lemm}
\label{ape-c2b-lemma}
There exist some uniformly positive constants $B, \delta, \e>0$ such that
\[h=(v-\ubar{v})+B\lt(\frac{1}{\sqrt{1-r^2}}-x_{n+1}\rt)\]
satisfying $\mathfrak{L}h\leq-a(1+\sum\limits_iF^{ii})$ in $U_{r\delta}$ and $h\geq 0$ on $\p U_{r\delta}.$
Here $a>0$ is some positive constant, $\ubar{v}=\frac{\ubar{u}^{r*}}{\sqrt{1-|\xi|^2}}$ is a subsolution,
$\mathfrak{L}f:= F^{ij}\nabla_{ij}f-f\sum\limits_iF^{ii},$ and $U_{r\delta}:=\lt\{x\in U_r\mid\frac{1}{\sqrt{1-r^2}}-x_{n+1}<\delta\rt\}.$
\end{lemm}
Following the argument in \cite{Guan99}, we obtain a $C^2$ boundary estimate for $u^{r*}.$ So far, we have obtained the $C^0,$ $C^1$, and $C^2$ boundary
estimates for the solution of \eqref{skc0.2}. To prove the solvability of \eqref{skc0.2}, we only need to
obtain the $C^2$ global estimates. We consider
\be\label{ape0.2}
\hat{F}=\lt(\frac{\s_n}{\s_{n-k}}\rt)^{\frac{1}{k}}(\Lambda_{ij})=(-v)^{-\frac{\al}{k}}:=\tilde{\Psi},
\ee
where $\Lambda_{ij}=v_{ij}-v\delta_{ij}.$
\begin{lemm}
\label{ape-c2g-lem}
Let $v$ be the solution of \eqref{ape0.2} in a bounded domain $U\subset\mathbb{H}^n.$ Denote the eigenvalues of $(v_{ij}-v\delta_{ij})$ by
$\lambda[v_{ij}-v\delta_{ij}]=(\lambda_1, \cdots, \lambda_n).$ Then
\[\lambda_{\max}\leq\max\{C, \lambda|_{\p U}\},\] and $C$ is a positive constant only depending on $U$ and $\tilde{\Psi}.$
\end{lemm}
\begin{proof}The proof of this Lemma is a modification of the proof of Lemma 18 in \cite{WX20}.
Set $M=\max\limits_{p\in\bar{U}}\max\limits_{|\xi|=1, \xi\in T_p\mathbb{H}^n}(\log\Lambda_{\xi\xi})+Nx_{n+1},$
where $x_{n+1}$ is the coordinate function. Without loss of generality, we may assume $M$ is achieved at an interior point $p_0\in U$ for some direction
$\xi_0.$ Choose an orthonormal frame $\{e_1, \cdots, e_n\}$ around $p_0$ such that $e_1(p_0)=\xi_0$ and $\Lambda_{ij}(p_0)=\lambda_i\delta_{ij}.$
Now, lets consider the test function
\[\phi=\log\Lambda_{11}+Nx_{n+1}.\]
At its maximum point $p_0,$ we have
\be\label{ape0.3}
0=\phi_i=\frac{\Lambda_{11i}}{\Lambda_{11}}+N(x_{n+1})_i,
\ee
and
\be\label{ape0.4}
0\geq\hat{F}^{ii}\phi_{ii}=\frac{\hat{F}^{ii}\Lambda_{11ii}}{\Lambda_{11}}-\hat{F}^{ii}\lt(\frac{\Lambda_{11i}}{\Lambda_{11}}\rt)^2+N(x_{n+1})\sum\limits_i\hat{F}^{ii}
\ee
Since $\Lambda_{11ii}=\Lambda_{ii11}+\Lambda_{ii}-\Lambda_{11}$ and
\[\hat{F}_{11}=\hat{F}^{ii}\Lambda_{ii11}+\hat{F}^{pq, rs}\Lambda_{pq1}\Lambda_{rs1}=\tilde{\Psi}_{11},\]
we get
\be\label{ape0.5}
\begin{aligned}
\hat{F}^{ii}\Lambda_{11ii}&=\hat{F}^{ii}\Lambda_{ii11}+\tilde{\Psi}-\Lambda_{11}\sum\limits_i\hat{F}^{ii}\\
&=\tilde{\Psi}_{11}-\hat{F}^{pp, qq}\Lambda_{pp1}\Lambda_{qq1}-\sum\limits_{p\neq q}\frac{\hat{F}^{pp}-\hat{F}^{qq}}{\lambda_p-\lambda_q}\Lambda^2_{pq1}
+\tilde{\Psi}-\Lambda_{11}\sum\limits_i\hat{F}^{ii}.
\end{aligned}
\ee
Since $\hat{F}$ is concave, combining \eqref{ape0.5} and \eqref{ape0.4} we have
\be\label{ape0.6}
\begin{aligned}
0&\geq\frac{1}{\Lambda_{11}}
\lt\{\tilde{\Psi}_{11}+2\sum\limits_{i\geq 2}\frac{\hat{F}^{ii}-\hat{F}^{11}}{\lambda_1-\lambda_i}\Lambda_{11i}^2+\tilde{\Psi}-\Lambda_{11}\sum\limits_i\hat{F}^{ii}\rt\}\\
&-\frac{\hat{F}^{ii}\Lambda^2_{11i}}{\Lambda_{11}^2}+Nx_{n+1}\sum\limits_i\hat{F}^{ii}.
\end{aligned}
\ee
We need an explicit expression of $\hat{F}^{ii}.$ A straightforward calculation gives
\[k\hat{F}^{k-1}\hat{F}^{ii}=\frac{\s_{n-1}(\la|i)}{\s_{n-k}}-\frac{\s_n}{\s_{n-k}^2}\s_{n-k-1}(\la|i).\]
Since
\[
\begin{aligned}
&\s_{n-1}(\la|i)\s_{n-k}-\s_n\s_{n-k-1}(\la|i)\\
&=\s_{n-1}(\la|i)[\la_i\s_{n-k-1}(\la|i)+\s_{n-k}(\la|i)]-\s_n\s_{n-k-1}(\la|i)\\
&=\s_{n-1}(\la|i)\s_{n-k}(\la|i),
\end{aligned}
\]
we get $$k\hat{F}^{k-1}\hat{F}^{ii}=\frac{\s_{n-1}(\la|i)\s_{n-k}(\ka|i)}{\s^2_{n-k}}.$$
Therefore, we have
\[
\begin{aligned}
&k\hat{F}^{k-1}(\hat{F}^{ii}-\hat{F}^{11})\\
&=\frac{1}{\s^2_{n-k}}[\s_{n-1}(\lambda|i)\s_{n-k}(\lambda|i)-\s_{n-1}(\lambda|1)\s_{n-k}(\lambda|1)]\\
&=\frac{1}{\s^2_{n-k}}[\s_{n-2}(\lambda|i1)\lambda_1\s_{n-k}(\lambda|i)-\s_{n-2}(\lambda|1i)\lambda_i\s_{n-k}(\lambda|1)]\\
&=\frac{\s_{n-2}(\lambda|1i)}{\s^2_{n-k}}[\lambda_1\s_{n-k}(\lambda|i)-\lambda_i\s_{n-k}(\lambda|1)]\\
&=\frac{\s_{n-2}(\lambda|1i)(\lambda_1-\lambda_i)}{\s^2_{n-k}}[(\lambda_1+\lambda_i)\s_{n-k-1}(\lambda|1i)+\s_{n-k}(\lambda|1i)]
\end{aligned}
\]
When $i\geq 2$ we can see that
\[
\begin{aligned}
&k\hat{F}^{k-1}\lt(\frac{\hat{F}^{ii}-\hat{F}^{11}}{\lambda_1-\lambda_i}-\frac{\hat{F}^{ii}}{\lambda_1}\rt)\\
&=\frac{\s_{n-2}(\lambda|1i)}{\s^2_{n-k}}[(\lambda_1+\lambda_i)\s_{n-k-1}(\lambda|1i)+\s_{n-k}(\lambda|1i)-\s_{n-k}(\lambda|i)]\\
&=\frac{\s_{n-1}(\lambda|1)}{\s^2_{n-k}}\s_{n-k-1}(\lambda|1i)>0
\end{aligned}
\]
Thus, \eqref{ape0.6} can be reduced to
\be\label{ape0.7}
\begin{aligned}
0&\geq\frac{1}{\Lambda_{11}}\tilde{\Psi}_{11}+(Nx_{n+1}-1)\sum\limits_i\hat{F}^{ii}-\frac{\hat{F}^{11}\Lambda^2_{111}}{\Lambda^2_{11}}\\
&=\frac{\tilde{\Psi}_{11}}{\Lambda_{11}}+(Nx_{n+1}-1)\sum\limits_i\hat{F}^{ii}-\hat{F}^{11}N^2(x_{n+1})^2_1.
\end{aligned}
\ee
Since $\tilde{\Psi}=(-v)^{-\frac{\al}{k}}$ and $-v=\frac{|u^*|}{\sqrt{1-|\xi|^2}}>\min\limits_{\xi\in \bar{B}_r}|u^*|>\lt|\max\limits_{\xi\in\p B_1}\varphi^*\rt|>0,$ a direct calculation yields
\[
\begin{aligned}
\tilde{\Psi}_{11}&=\frac{\al}{k}\lt(\frac{\al}{k}+1\rt)(-v)^{-\frac{\al}{k}-2}v^2_1+\frac{\al}{k}(-v)^{-\frac{\al}{k}-1}v_{11}\\
&\geq\frac{\al}{k}(-v)^{-\frac{\al}{k}-1}(\lambda_1+v)\\
&\geq C_1\lambda_1- C_2.
\end{aligned}
\]
Here, $C_1$ depends on $U$, since $-v\leq\frac{C}{\sqrt{1-|\xi|^2}}$. Plugging the above inequality into \eqref{ape0.7} we obtain
\[0\geq C_1-\frac{C_2}{\lambda_1}+(Nx_{n+1}-1)\sum\limits_i\hat{F}^{ii}-N^2(x_{n+1})^2_1\frac{C_3}{\lambda_1}.\]
Here we have used
\[k\hat{F}^{k-1}\hat{F}^{11}=\frac{\s_n\s_{n-k}(\la|1)}{\la_1\s_{n-k}^2}<\frac{1}{\la_1}\hat{F}^k\leq\frac{C_3}{\lambda_1},\]
where $C_3$ depends on $U$. Let $N=2$ we can see that when $\lambda_1$ is large, we get an contradiction. This completes the proof of Lemma \ref{ape-c2g-lem}.
\end{proof}
Therefore, we conclude that the approximate problem \eqref{skc0.2} is solvable.

\subsection{Local a priori estimates}
Let $u^{r*}$ be the solution of \eqref{skc0.2}, $u_r$ be the Legendre transform of $u^{r*}.$ In this section, we will study interior estimates of $u_r,$ which will enable us
to show there exists a subsequence of $\{{u_r\}}$ that converges to the desired entire solution $u$ of \eqref{se1}.
\label{lae}
\begin{lemm}(Lemma 5.1 of \cite{BP})
\label{lae-lem0}
Let $\Omega\subset\R^n$ be a bounded open set. Let $u, \bar{u}, \Psi: \Omega\goto\R^n$ be strictly spacelike. Assume that $u$ is strictly convex and
$u<\bar{u}$ in $\Omega.$ Also assume that near $\partial\Omega,$ we have $\Psi>\bar{u}.$ Consider the set where $u>\Psi.$ For every $x$ in this set, we have the following
gradient estimate for $u$:
\[\frac{1}{\sqrt{1-|Du|^2}}\leq\frac{1}{u(x)-\Psi(x)}\cdot\sup\limits_{\{u>\Psi\}}\frac{\bar{u}-\Psi}{\sqrt{1-|D\Psi|^2}}.\]
\end{lemm}

\subsubsection{Construction of $\Psi$} In order to obtain the local $C^1$ estimate, we introduce a new subsolution $\ubar{u}_1$ of \eqref{se1},
where $\ubar{u}_1$ satisfies
\[\s_n(\ka[\M_{\ubar{u}_1}])=100\lt(-\lt<X,\nu\rt>\rt)^{\frac{\al n}{k}},\] and
\[\ubar{u}_1(x)-|x|\goto\varphi\lt(\frac{x}{|x|}\rt)\,\,\mbox{as $|x|\goto\infty.$}\]
\begin{lemm}
\label{lae-lem1} Let $\ubar{u}$ be a solution of
\[\s_n(\ka[\M_{u}])=\frac{1}{{n\choose k}^{\frac{n}{k}}}\lt(-\lt<X, \nu\rt>\rt)^{\frac{\al n}{k}}\]
satisfying $\ubar{u}(x)-|x|\goto\varphi\lt(\frac{x}{|x|}\rt)$ as $|x|\goto\infty,$
then $\ubar{u}_1<\ubar{u}.$
\end{lemm}
\begin{proof}
We look at the Legendre transform of $\ubar{u}_1,$ denoted by $\lus_1.$ Then $\lus_1$ satisfies
\[\s_n(w^*\gas_{ik}(\lus_1)_{kl}\gas_{lj})=\frac{1}{100}\lt(\frac{\sqrt{1-|\xi|^2}}{-\lus_1}\rt)^\frac{\al n}{k};\]
while $\lus$ satisfies
\[\s_n(w^*\gas_{ik}(\lus)_{kl}\gas_{lj})={n\choose k}^{\frac{n}{k}}\lt(\frac{\sqrt{1-|\xi|^2}}{-\lus}\rt)^\frac{\al n}{k}.\]
Moreover, $\lus_1=\lus=\varphi^*(\xi)$ on $\p B_1.$
Applying the maximal principle we conclude $\lus_1>\lus$ in $B_1.$
Following the proof of Lemma 13 of \cite{WX20} we get $\ubar{u}_1<\ubar{u}$ in $\R^n.$
\end{proof}
Now, for any compact domain $K\subset\R^n,$ let $2\delta=\min\limits_{K}(\ubar{u}-\ubar{u}_1).$ We define
$\Psi=\ubar{u}_1+\delta.$ Denote $K'=\{x\in\R^n\mid \Psi\leq\bar{u}\},$ notice that as $|x|\goto\infty,$ we have
$\ubar{u}_1-\bar{u}\goto 0,$ this implies $K'$ is compact. Applying Lemma \ref{lae-lem0}, for any $(\Omega_r, u^r),$
if $K'\subset\Omega_r$ we have
\[\sup\limits_{K}\frac{1}{\sqrt{1-|Du^r|^2}}\leq\frac{1}{\delta}\sup\limits_{K'}\frac{\bar{u}-\Psi}{\sqrt{1-|D\Psi|^2}}.\]

\subsubsection{Local $C^2$ estimates}
We will follow the proof of Lemma 24 in \cite{WX20}.
\begin{lemm}
\label{lc2lem1}
Let $u^{r*}$ be the solution of \eqref{skc0.2}, $u_r$ be the Legendre transform of $u^{r*},$ and $\Omega_r=Du^{r*}(B_r).$
For any giving $s>1,$ let $r_s>0$ be a positive number such that when $r>r_s,$ $u_r|_{\p\Omega_r}>s.$ Let $\ka_{\max}(x)$
be the largest principal curvature of $\M_{u_r}$ at $x,$ where $\M_{u_r}=\{(x, u_r(x))|x\in\Omega_r\}.$
Then, for $r>r_s$ we have
\[\max\limits_{\{x\in\Omega_r|u_r\leq s\}}(s-u_r)\ka_{\max}\leq C.\]
Here, $C$ only depends on the $C^0$ and local $C^1$ estimates of $u_r$.
\end{lemm}
\begin{proof}
Consider the test function
\be\label{lae2.1}
\phi=m\log(s-u)+\log P_m-mN\lt<\nu, E\rt>,
\ee
where $P_m=\sum\limits_j\ka_j^m,$ $E=(0, \cdots, 0, 1),$ and $N, m>0$ are some undetermined constants.
Assume that $\phi$ achieves its maximum value on $\M$ at some point $x_0.$ We may choose a local orthonormal frame $\{\tau_1, \cdots, \tau_n\}$ such that
at $x_0,$ $h_{ij}=\ka_i\delta_{ij}$ and $\ka_1\geq\ka_2\geq\cdots\geq\ka_n.$ Differentiating $\phi$ twice at $x_0$ we have
\be\label{lae2.2}
\frac{\sum\limits_j\ka_j^{m-1}h_{jji}}{P_m}-Nh_{ii}\lt<\tau_i, E\rt>+\frac{\lt<\tau_i, E\rt>}{s-u}=0,
\ee
and
\be\label{lae2.3}
\begin{aligned}
0&\geq\frac{1}{P_m}\lt[\sum\limits_j\ka_j^{m-1}h_{jjii}+(m-1)\sum\limits_j\ka_j^{m-2}h_{jji}^2+\sum\limits_{p\neq q}\frac{\ka_p^{m-1}-\ka_q^{m-1}}{\ka_p-\ka_q}h^2_{pqi}\rt]\\
&-\frac{m}{P_m^2}\lt(\sum\limits_j\ka_j^{m-1}h_{jji}\rt)^2-Nh_{iil}\lt<\tau_l, E\rt>-Nh_{ii}^2\lt<\nu, E\rt>+\frac{h_{ii}\lt<\nu, E\rt>}{s-u}-\frac{u_i^2}{(s-u)^2}.
\end{aligned}
\ee
Denote $\hat{v}=-\lt<X, \nu\rt>$ then
\[\hat{v}_j=-h_{jk}\lt<X, \tau_k\rt>=-h_{jj}\lt<X, \tau_j\rt>,\]
and
\[
\begin{aligned}
\hat{v}_{jj}&=-h_{jjk}\lt<X, \tau_k\rt>-h_{jk}\lt<\tau_j, \tau_k\rt>-h^2_{jk}\lt<X, \nu\rt>\\
&=-h_{jjk}\lt<X, \tau_k\rt>-h_{jj}-h^2_{jj}\lt<X, \nu\rt>.
\end{aligned}
\]
Since $\s_k=\hat{v}^\al:=G,$ we can see that $\s_k^{ii}h_{iij}=G_j$ and
$\s_{k}^{ii}h_{iijj}+\s_k^{pq, rs}h_{pqj}h_{rsj}=G_{jj}.$ Recall also that in Minkowski space we have
\[h_{jjii}=h_{iijj}+h^2_{ii}h_{jj}-h_{ii}h^2_{jj},\] thus \eqref{lae2.3} becomes
\be\label{lae2.4}
\begin{aligned}
0&\geq\frac{1}{P_m}\bigg[\sum\limits_j\ka_j^{m-1}\s_k^{ii}(h_{iijj}+h^2_{ii}h_{jj}-h_{ii}h^2_{jj})\\
&+(m-1)\s_k^{ii}\sum\limits_j\ka_j^{m-2}h^2_{jji}+\sum\limits_{p\neq q}\frac{\ka_p^{m-1}-\ka_q^{m-1}}{\ka_p-\ka_q}\s_k^{ii}h^2_{pqi}\bigg]\\
&-\frac{m}{P_m^2}\s_k^{ii}\lt(\sum\limits_j\ka_j^{m-1}h_{jji}\rt)^2-N\lt<\nabla G, E\rt>-N\s_k^{ii}\ka_i^2\lt<\nu, E\rt>
+\frac{kG\lt<\nu, E\rt>}{s-u}-\frac{\s_k^{ii}u_i^2}{(s-u)^2}.
\end{aligned}
\ee
This gives
\be\label{lae2.5}
\begin{aligned}
0&\geq\frac{1}{P_m}\bigg\{\sum\limits_j\ka_j^{m-1}[G_{jj}-\s_k^{pq, rs}h_{pqj}h_{rsj}-kGh^2_{jj}]\\
&+(m-1)\s_k^{ii}\sum\limits_j\ka_j^{m-2}h^2_{jji}+\sum\limits_{p\neq q}\frac{\ka_p^{m-1}-\ka_q^{m-1}}{\ka_p-\ka_q}\s_k^{ii}h^2_{pqi}\bigg\}\\
&-\frac{m}{P_m^2}\s_k^{ii}\lt(\sum\limits_j\ka_j^{m-1}h_{jji}\rt)^2-N\lt<\nabla G, E\rt>-N\s_k^{ii}\ka_i^2\lt<\nu, E\rt>
+\frac{kG\lt<\nu, E\rt>}{s-u}-\frac{\s_k^{ii}u_i^2}{(s-u)^2}.
\end{aligned}
\ee

We denote $A_i=\frac{\ka_i^{m-1}}{P_m}\lt[K(\s_k)^2_i-\sum\limits_{p, q}\s_k^{pp, qq}h_{ppi}h_{qqi}\rt],$
$B_i=2\frac{\ka_j^{m-1}}{P_m}\sum\limits_j\s_k^{jj, ii}h^2_{jji},$\\ $C_i=\frac{m-1}{P_m}\s_k^{ii}\sum\limits_j\ka_j^{m-2}h^2_{jji},$
$D_i=\frac{2\s_k^{jj}}{P_m}\sum\limits_{j\neq i}\frac{\ka_j^{m-1}-\ka_i^{m-1}}{\ka_j-\ka_i}h^2_{jji},$
and $E_i=\frac{m}{P_m^2}\s_k^{ii}\lt(\sum\limits_j\ka_j^{m-1}h_{jji}\rt)^2.$ Then \eqref{lae2.5} can be reduced to
\be\label{lae2.6}
\begin{aligned}
0&\geq\sum\limits_i(A_i+B_i+C_i+D_i-E_i)-\sum\limits_i\frac{K\ka_i^{m-1}(G_i)^2}{P_m}\\
&+\frac{\sum_j\ka_j^{m-1}G_{jj}}{P_m}-N\lt<\nabla G, E\rt>-N\s_k^{ii}\ka_i^2\lt<\nu, E\rt>\\
&-\frac{\sum_j\ka_j^{m+1}}{P_m}kG+\frac{kG\lt<\nu, E\rt>}{s-u}-\frac{\s_k^{ii}u_i^2}{(s-u)^2}.
\end{aligned}
\ee
A straightforward calculation shows
\be\label{lae2.7}
\begin{aligned}
&\frac{\sum_j\ka_j^{m-1}G_{jj}}{P_m}=\frac{\sum_j\ka_j^{m-1}[\al(\al-1)\hat{v}^{\al-2}\hat{v}^2_j+\al\hat{v}^{\al-1}\hat{v}_{jj}]}{P_m}\\
&=\al(\al-1)\hat{v}^{\al-2}\frac{\sum_j\ka_j^{m-1}\hat{v}_j^2}{P_m}+\al\hat{v}^{\al-1}\frac{\sum_j\ka_j^{m-1}
\lt(h_{jjl}\lt<-X, \tau_l\rt>-\ka_j+\ka_j^2\hat{v}\rt)}{P_m}\\
&=\al(\al-1)\hat{v}^{\al-2}\frac{\sum_j\ka_j^{m-1}\hat{v}_j^2}{P_m}+\al\hat{v}^{\al-1}\frac{\sum_j\ka_j^{m-1}h_{jjl}\lt<-X, \tau_l\rt>}{P_m}\\
&-\al\hat{v}^{\al-1}+\al\hat{v}^\al\frac{\sum_j\ka_j^{m+1}}{P_m}.
\end{aligned}
\ee
Moreover, we have
\be\label{lae2.8}
\begin{aligned}
&\frac{\al\hat{v}^{\al-1}\sum\ka_j^{m-1}h_{jjl}\lt<-X, \tau_l\rt>}{P_m}-N\lt<\nabla G, E\rt>\\
&=\frac{\al\hat{v}^{\al-1}\sum\ka_j^{m-1}h_{jjl}\lt<-X, \tau_l\rt>}{P_m}-N\al\hat{v}^{\al-1}\hat{v}_l\lt<\tau_l, E\rt>\\
&=\al\hat{v}^{\al-1}\lt(\frac{\sum\ka_j^{m-1}h_{jjl}\lt<-X, \tau_l\rt>}{P_m}-N\ka_l\lt<X, \tau_l\rt>u_l\rt)\\
&=\al\hat{v}^{\al-1}\sum\lt<X, \tau_l\rt>\lt(N\ka_lu_l-\frac{u_l}{s-u}-N\ka_lu_l\rt)\\
&=-\frac{\al\hat{v}^{\al-1}\sum\lt<X, \tau_l\rt>u_l}{s-u},
\end{aligned}
\ee
where we have used \eqref{lae2.2}.
Combing \eqref{lae2.7}, \eqref{lae2.8} with \eqref{lae2.6} we obtain
\be\label{lae2.9}
\begin{aligned}
0&\geq\sum\limits_i(A_i+B_i+C_i+D_i-E_i)-\sum\limits_i\frac{K\ka_i^{m-1}(G_i)^2}{P_m}\\
&+\al(\al-1)\hat{v}^{\al-2}\frac{\sum_j\ka_j^{m-1}\hat{v}_j^2}{P_m}-\al\hat{v}^{\al-1}+\al\hat{v}^\al\frac{\sum_j\ka_j^{m+1}}{P_m}\\
&-\frac{\al\hat{v}^{\al-1}\sum_l\lt<X, \tau_l\rt>u_l}{s-u}-N\s_k^{ii}\ka_i^2\lt<\nu, E\rt>\\
&-kG\frac{\sum_j\ka_j^{m+1}}{P_m}+\frac{kG\lt<\nu, E\rt>}{s-u}-\frac{\s_k^{ii}u_i^2}{(s-u)^2}.
\end{aligned}
\ee
Recall that
$$\langle X,X\rangle+\langle \nu, X\rangle^2=\sum_i\langle X,\tau_i\rangle^2,$$ we know
$|\langle X,\tau_i\rangle|$ can be controlled by some constants depending on $s$ and local $C^1$ estimates.
Therefore, applying Lemma 8 and 9 of \cite{LRW16} we may assume
\be\label{lae2.10}
\begin{aligned}
0&\geq-C\ka_1+\sum\limits_{i=2}^n\frac{\s_k^{ii}}{P_m^2}\lt(\sum_j\ka_j^{m-1}h_{jji}\rt)^2-\frac{C}{s-u}\\
&-N\s_k^{ii}\ka_i^2\lt<\nu, E\rt>+\frac{kG\lt<\nu, E\rt>}{s-u}-\frac{\s_k^{ii}u_i^2}{(s-u)^2}.
\end{aligned}
\ee

Now, for any fixed $i\geq 2$ by \eqref{lae2.2} we have
\be\label{lae2.11}
\begin{aligned}
&\frac{\s_k^{ii}u_i^2}{(s-u)^2}=\s_k^{ii}\lt[\frac{\sum\ka_j^{m-1}h_{jji}}{P_m}+N\ka_iu_i\rt]^2\\
&=\s_k^{ii}\lt(\frac{\sum\ka_j^{m-1}h_{jji}}{P_m}\rt)^2+2N\s_k^{ii}\ka_iu_i\lt(-N\ka_iu_i+\frac{u_i}{s-u}\rt)+N^2\s_k^{ii}\ka_i^2u_i^2\\
&=\s_k^{ii}\lt(\frac{\sum\ka_j^{m-1}h_{jji}}{P_m}\rt)^2-N^2\s_k^{ii}\ka_i^2u_i^2+2N\frac{\s_k^{ii}\ka_iu_i^2}{s-u}
\end{aligned}
\ee

Plugging \eqref{lae2.11} into \eqref{lae2.10}
we get,
\[\begin{aligned}
0&\geq-C\ka_1-\frac{C}{s-u}-N\s_k^{ii}\ka_i^2\lt<\nu, E\rt>+\frac{kG\lt<\nu, E\rt>}{s-u}\\
&-\frac{\s_k^{11}u_1^2}{(s-u)^2}+\sum\limits_{i=2}^nN^2\s_k^{ii}\ka_i^2u_i^2-2N\sum\limits_{i=2}^n\frac{\s_k^{ii}\ka_iu_i^2}{s-u}
\end{aligned}\]
Since there is some constant $c_0$ such that $\s_k^{11}\ka_1\geq c_0>0,$ we have
\[\begin{aligned}
0&\geq\lt(-\frac{c_0N\lt<\nu, E\rt>}{2}-C\rt)\ka_1-\frac{N}{2}\s_k^{11}\ka_1^2\lt<\nu, E\rt>\\
&-\sum\limits_{i=2}^n\frac{2N\s_ku_i^2}{s-u}+\frac{kG\lt<\nu, E\rt>-C}{s-u}-\frac{\s_k^{11}u_1^2}{(s-u)^2},
\end{aligned}\]
where we have used for any $1\leq i\leq n$ (no summation), $\s_k=\s_k^{ii}\ka_i+\s_k(\ka|i)\geq\s_k^{ii}\ka_i.$
Moreover, it's clear that
\[\sum\limits_{i=2}^nu_i^2=\sum\limits_{i=2}^n\lt<\tau_i, E\rt>^2<\frac{1}{1-|Du|^2}=\lt<\nu, E\rt>^2.\]
We conclude
\[\lt(\frac{2NC}{s-u}+\frac{\s_k^{11}}{(s-u)^2}\rt)\lt<\nu, E\rt>^2\geq\frac{Nc_0\ka_1}{4}\lt<-\nu, E\rt>+\frac{N}{2}\s_k^{11}\ka_1^2\lt<-\nu, E\rt>.\]
This implies $(s-u)\ka_1\leq C,$ where $C$ depends on $s$ and local $C^1$ estimates. Therefore, we obtain the desired Pogorelov type $C^2$ local estimates.
\end{proof}
Following the argument in subsection 6.4 of \cite{WX20}, we prove Theorem \ref{int-thm2}.


\begin{thebibliography}{99}
\bibitem{AW94}
Altschuler, Steven J.; Wu, Lang F.
{\em Translating surfaces of the non-parametric mean curvature flow with prescribed contact angle.}
Calc. Var. Partial Differential Equations 2 (1994), no. 1, 101-111.

\bibitem{BP}
Bayard, Pierre; Schn\"{u}rer, Oliver C.
{Entire spacelike hypersurfaces of constant Gauß curvature in Minkowski space.}
J. Reine Angew. Math. 627 (2009), 1-29.

\bibitem{CDL21}
Choi, Kyeongsu; Daskalopoulos, Panagiota; Lee, Ki-Ahm
{\em Translating solutions to the Gauss curvature flow with flat sides.}
Anal. PDE 14 (2021), no. 2, 595-616.

\bibitem{CNS1}
Caffarelli, L.; Nirenberg, L.; Spruck, J.
{\em The Dirichlet problem for nonlinear second-order elliptic equations. I. Monge-Ampère equation.}
Comm. Pure Appl. Math. 37 (1984), no. 3, 369-402.

\bibitem{CNS3}
Caffarelli, L.; Nirenberg, L.; Spruck, J.
{\em The Dirichlet problem for nonlinear second-order elliptic equations. III. Functions of the eigenvalues of the Hessian.}
Acta Math. 155 (1985), no. 3-4, 261-301.

\bibitem{CSS07}
Clutterbuck, Julie; Schn\"{u}rer, Oliver C.; Schulze, Felix
{\em Stability of translating solutions to mean curvature flow.}
Calc. Var. Partial Differential Equations 29 (2007), no. 3, 281-293.

\bibitem{Guan99}
Guan, Bo
{\em The Dirichlet problem for Hessian equations on Riemannian manifolds.}
Calc. Var. Partial Differential Equations 8 (1999), no. 1, 45-69.

\bibitem{GJS}
Guan, Bo; Jian, Huai-Yu; Schoen, Richard M.
{Entire spacelike hypersurfaces of prescribed Gauss curvature in Minkowski space.}
J. Reine Angew. Math. 595 (2006), 167-188.

\bibitem{HIMW19}
Hoffman, D.; Ilmanen, T.; Martín, F.; White, B.
{\em Graphical translators for mean curvature flow.}
Calc. Var. Partial Differential Equations 58 (2019), no. 4, Paper No. 117, 29 pp.

\bibitem{AML}
Li, An Min
{\em Spacelike hypersurfaces with constant Gauss-Kronecker curvature in the Minkowski space.}
Arch. Math. (Basel) 64 (1995), no. 6, 534-551.

\bibitem{LRW16}
Li, Ming; Ren, Changyu; Wang, Zhizhang
{\em An interior estimate for convex solutions and a rigidity theorem.}
J. Funct. Anal. 270 (2016), no. 7, 2691-2714.

\bibitem{RWX19}Ren, Changyu; Wang, Zhizhang; Xiao, Ling
{\em Entire spacelike hypersurfaces with constant $\sigma_{n-1}$ curvature in Minkowski space.}
eprint arXiv:2005.06109.

\bibitem{RWX20}
Ren, Changyu; Wang, Zhizhang; Xiao, Ling
{\em The prescribed curvature problem for entire hypersurfaces in Minkowski space.}
eprint arXiv:2007.04493.

\bibitem{SW10}
Sheng, WeiMin; Wu, Chao
{\em Rotationally symmetric translating soliton of Hk-flow.}
Sci. China Math. 53 (2010), no. 4, 1011-1016.

\bibitem{SX20}
Spruck, Joel; Xiao, Ling
{\em Complete translating solitons to the mean curvature flow in R3 with nonnegative mean curvature.}
Amer. J. Math. 142 (2020), no. 3, 993-1015.

\bibitem{Urb98}
Urbas, John
{\em Complete noncompact self-similar solutions of Gauss curvature flows. I. Positive powers.}
Math. Ann. 311 (1998), no. 2, 251-274.

\bibitem{WX20}
Wang, Zhizhang; Xiao, Ling;
{\em Entire spacelike hypersurfaces with constant $\s_k$ curvature in Minkowski space.}
Math. Ann. 382 (2022), no. 3-4, 1279-1322.

\bibitem{WXflow} Wang, Zhizhang; Xiao, Ling
{\em Entire convex curvature flow in Minkowski space.}
Preprint.

\bibitem{WX22}
Wang, Zhizhang; Xiao, Ling
{\em Entire $\sigma_k$ curvature flow in Minkowski space.}
Preprint.

\end{thebibliography}
\end{document}